\documentclass[12pt]{article}
\usepackage{amsfonts,amsthm}
\usepackage{amssymb,amsmath}
\usepackage{enumerate}
\usepackage{amsmath}

\usepackage[dvipsnames]{xcolor}
\usepackage{hyperref}
\hypersetup{
    colorlinks=true,
    citecolor=blue,
    linkcolor=blue,
    filecolor=magenta,
    urlcolor=blue,
    pdftitle={Overleaf Example},
    pdfpagemode=FullScreen,
    }

\input{amssym.def}

\newtheorem{Definition}{Definition}[section]
\newtheorem{Theorem}[Definition]{Theorem}
\newtheorem{Lemma}[Definition]{Lemma}
\newtheorem{Proposition}[Definition]{Proposition}
\newtheorem{Corollary}[Definition]{Corollary}
\newtheorem{Example}[Definition]{Example}
\newtheorem{Remark}[Definition]{Remark}

\newcommand{\be}{\begin{equation}}
\newcommand{\ee}{\end{equation}}

\begin{document}

\title{\bf Quasi S-n-ideals in commutative semirings}
\author{\bf Amaresh Mahato ${^{1}}$ \footnote {e-mail : amaresh.blg2015@gmail.com }, \bf  Saikat Das ${^{1}}$ \footnote {e-mail : saikatofficial607@gmail.com}, \bf Manasi Mandal ${^{1}}$\footnote {e-mail : manasi$_{-}$ju@yahoo.in}, \bf Sampad Das ${^{2}}$ \footnote {e-mail : jumathsampad@gmail.com}\ \\
{\small ${^{1}}$ Department of Mathematics, Jadavpur University, Kolkata -- 700032, India}\\
{\small ${^ { 2 }}$ Department of Mathematics, Kanyashree University, Nadia-741101, India.}} 
\date{}
\maketitle

\begin{abstract}

 Let $R$ be a commutative semiring with unity $(1\neq0)$, and let $S$ be a proper multiplicatively closed subset of $R$. In this paper, we introduce the notion of quasi $S$-$n$-ideals in commutative semirings. We the study basic property of quasi $S$-$n$-ideals and establish their relationships with quasi $n$-ideals, $S$-$n$-ideals, $S$-prime ideals, and $S$-primary ideals. We also investigate their behavior under localization. Finally, we determine quasi $S$-$n$-ideals in the quotient polynomial semiring $R[x]/\langle x^m\rangle$ and study their natural extensions via the constructions of idealization and amalgamation.

 

\end{abstract}

\section{Introduction}
All semirings are taken to be commutative with nonzero identity throughout this work. Let $R$ be a semiring. A subset $S$ of $R$ is called a multiplicatively closed set (m.c.s.) if $0\notin S$, $1\in S$, and
$st\in S$ for all $s,t\in S$. The sets of units, regular elements, zero-divisors and nilpotent elements of $R$ are denoted by $U(R)$, $Reg(R)$, $Z(R)$, and $N(R)$, respectively. The radical of an ideal $I$ is represented by $\sqrt{I}$ and defined as $\sqrt{I}=\{a \in R : a^n \in I, n \in \mathbb{N} \}$. In particular, $\sqrt{(0)}=N(R)$, the nilradical of $R$.

Prime ideals and primary ideals are the most fundamental concepts in the ideal theory of rings and semirings.
Over the past few years, these notions have been generalized in several directions. One important generalization is the concept of an $S$-prime ideal, an ideal $I$ disjoint from $S$ is said to be $S$-prime ideal of $R$ if there exists an element $s \in S$ such that whenever $a,b \in R$ and $ab\in I $  either $sa\in I$ or $sb \in I$, introduced by Hamed and Malek \cite{hm}. Later, Massoud \cite{ma} introduced the notion of an $S$-primary ideal. Since then several related concepts, including quasi $S$-primary ideals, $S$-$k$-primary ideals and $S$-$k$-primary decompositions have been investigated by many authors \cite{ko,mo,ma,am}. These studies have led to the development of the rich theory  of $S$-ideals in commutative semirings, including localization results, decomposition theorems and characterizations of $S$-prime and $S$-primary ideals.

Another direction of generalization is based on nilpotent elements. In \cite{tk} introduced the notion of an $n$-ideal in ring theory.  The concept of $n$-ideals has been further generalized in ring theory, giving rise to several related notions such as quasi $n$-ideals, $S$-$n$-ideals and $J$-ideals \cite{am1,al,bo,kh}. In the setting of commutative semirings, H. K. Ranote \cite{ra} introduced the notion of $n$-ideal. A proper ideal $I$ of a semiring $R$ is called $n$-ideal whenever $a,b \in R$ and $ab \in I$ implies either $a\in I$ or $b\in N(R)$.



 Here we introduce the notion of a quasi $n$-ideal and quasi $S$-$n$-ideal as a common generalization of $n$-ideals. A proper ideal $I$ of a semiring $R$ is called a quasi $S$-$n$-ideal if there exists an element $s\in S$ such that for all $a,b\in R$, whenever $ab\in I$ implies either $sa^2\in I$ or $sb\in N(R)$. Thus the definition of a quasi $S$-$n$-ideal weakens that of an $S$-$n$-ideal by replacing the condition $sa\in I$ with the weaker condition $sa^2\in I$.
Quasi $S$-$n$-ideals provide a natural link between $S$-$n$-ideals and $S$-primary ideals. Every $S$-$n$-ideal is a quasi $S$-$n$-ideal, whereas the converse does not hold in general. This naturally raises the question of which properties of $S$-$n$-ideals extend to quasi $S$-$n$-ideals and how these ideals are related to other important classes of ideals in commutative semirings. Motivated by these questions, we undertake a systematic study of quasi $S$-$n$-ideals and investigate their algebraic and structural properties. Also motivated by the work Paykan \cite{pk} on ideals in quotient polynomial ring, we study the corresponding relationships between ideals in quotient polynomial semirings.


Throughout this paper, we study quasi $S$-$n$-ideals in commutative semirings. In section~\ref{9}, we study basic properties and characterizations of quasi $S$-$n$-ideals. In particular, we prove that if $I$ is a quasi $S$-$n$-ideal of $R$, then there exists an element $s \in S$ such that $s\sqrt{I}\subseteq N(R)$ ($cf.$ Proposition \ref{2.3}).  We also show that when  $\sqrt{I}$ is a quasi $S$-$n$-ideal of $R$ and $\sqrt{I}=N(R)$, the $\sqrt{I}$ is an $S$-prime ideal of $R$ ($cf.$ Theorem \ref{2.10}). We further investigate the relationships between quasi $S$-$n$-ideals and several well-known classes of ideals, including $S$-$n$-ideals, $S$-prime ideals and $S$-primary ideals. Furthermore, we study their behavior under localization and prove that under the condition $N(S^{-1}R)=S^{-1}N(R)$, a proper ideal $I$ of $R$ is a quasi $S$-$n$-ideal if and only if $S^{-1}I$ is a quasi $n$-ideal of $S^{-1}R$ ($cf.$ Theorem \ref{8}). We then consider quotient polynomial semirings and establish that the ideal $J=I+Rx+Rx^2+\cdots+Rx^{m-1}+\langle x^m\rangle$ of $R[x]/\langle x^m\rangle$ is a quasi $\overline{S}$-$n$-ideal if and only if $I$ is a quasi $S$-$n$-ideal of $R$ ($cf.$ Theorem \ref{12}). In section \ref{2}, we investigate quasi $S$-$n$-ideals under idealization and amalgamation and obtain corresponding results under suitable assumptions.



\section {Quasi $S$-$n$-ideal} \label{9}

\begin{Definition} 
Let $R$ be a commutative semiring and $I$ be a proper ideal of $R$. Then $I$ is called a quasi $n$-ideal of $R$ if for every $a,b \in R$, $ab\in I \Longrightarrow a^{2}\in I \ \text{or}\ b\in N(R)$.
\end{Definition}

\begin{Definition} \label{2.2}
Let $R$ be a commutative semiring and let $S$ be a multiplicatively closed subset of $R$. A proper ideal $I$ of $R$ satisfying $I\cap S=\phi$ is called a quasi $S$-$n$-ideal of $R$ if there exists a fixed element $s \in S$ such that, whenever $a,b\in R$ and $ab\in I$, either $sa^{2}\in I$ or $sb\in N(R)$, where $N(R)=\{x\in R\mid x^{m}=0\text{ for some }m\in\mathbb{N}\}$ denotes the nilradical of $R$.

The fixed element $s$ is called quasi $S$-$n$ element of $I$.
\end{Definition}

Clearly, when $S=\{1\}$, quasi-$S$-$n$-ideals are precisely the quasi $n$-ideals.

\begin{Example}
$(i)$ Let $R=\mathbb{Z}_0^+$ be the semiring of non-negative integers and let $I=(0)$ be the zero ideal of $R$. The nilradical of $R$ is $N(R)=\{0\}$. Let $a,b\in R$ with $ab\in I$. Then $ab=0$, which implies either $a=0$ or $b=0$. If $a=0$, then $sa^{2}=0\in I$. If $b=0$, then $sb\in N(R)$. Hence, $ab\in I \Longrightarrow sa^{2}\in I \text{ or } sb\in N(R)$, and therefore $I$ is a quasi $S$-$n$-ideal of $R$  for any multiplicatively closed subset $S$ of $R$.

   $ (ii)$ Let $R=\mathbb{Z}_{12}^+$ and let $S=\{3^n : n\geq 0\}=\{1,3,9\}$. Let an ideal $I=(2)$. Then $I \cap S = \phi$. The nilradical of $R$ is $N(R)=\{0,6\}$. Take $s=3 \in S$. If $ab \in I$, then $ab$ is even. Hence either $a$ is even or $b$ is even. If $a$ is even, then $a^2$ is even and $3a^2 \in I$. If $b$ is even, then $3b \in \{0,6\}=N(R)$. Thus $I$ is a quasi $S$-$n$-ideal but not a quasi $n$-ideal of $R=\mathbb{Z}_{12}^+$. Since $2 \cdot 3=6 \in I$ but $3^2 \in I$ and $2 \notin N(R)$.

  $ (iii)$  If we take  $R=\mathbb{Z}_{p^k q}^+$ and $S=\{q^n : n\geq 0\}$ then quasi $S$-$n$-ideal are $(0)$ and $(p^n)$, for all $1 \le n \le k$, and quasi $n$- ideal for only $I=(0)$.

\end{Example}

\begin{Proposition}\label{2.3}
Let $R$ be a commutative semiring and let $S$ be a multiplicatively closed subset of $R$. Let $I$ be a quasi $S$-$n$-ideal of $R$, there exists $s\in S$ such that $s\sqrt{I} \subseteq N(R)$.
\end{Proposition}

\begin{proof}
Since $I$ is a quasi $S$-$n$-ideal, there exists an element $s \in S$ such that for all $a,b\in R$, $ab \in I$ implies $sa^2 \in I$ or $sb \in N(R)$. Let $x\in \sqrt{I}$. Then $x^n=1\cdot x^n\in I$. Therefore, there exists $s \in S$ such that $s \in I$ or $sx^n \in N(R)$. Since $I \cap S = \phi$, we have $s \notin I$. Therefore, $sx^n \in N(R) $ implies $ (sx)^n=s^{n-1} \cdot sx^n \in N(R)$ and so $sx \in N(R)$. Thus $s \sqrt{I} \subseteq N(R)$.
\end{proof}



\begin{Corollary}
Let $I$ be a quasi $n$-ideal of a commutative semiring $R$. Then $\sqrt{I}\subseteq N(R)$.
\end{Corollary}
\begin{proof}
Taking $S=\{1\}$ in Proposition \ref{2.3}, we obtain $\sqrt{I}\subseteq N(R)$.
\end{proof}

The converse of the above Proposition \ref{2.3}  may not be hold.

\begin{Example}
Let $R=\mathbb{Z}^+_{60}$ and $S=\{5^n:n\in\mathbb{N}_0\}=\{1,5,25\}$. Let $I=(6)$ and so $\sqrt{I}=(6)$. Then $I\cap S=\phi$. Since $60=2^2\cdot 3\cdot 5$, the nilradical of $R$ is $N(R)=\sqrt{(0)}=(30)=\{0,30\}$. Now, for $s=5\in S$ and any $x\in \sqrt{I}$, we have $x=6k$ for some $k\in\mathbb{Z}^+_{60}$. Hence
$sx=5(6k)=30k\in (30)=N(R)$. Therefore, $5\sqrt{I} \subseteq N(R)$. However, $I=(6)$ is not a quasi $S$-$n$-ideal of $R$. Indeed, take $a=2$ and $b=3$. Then $ab=6\in I$. But there is no element $s \in S$ such that $sa^2 \in I$ or $sb \in N(R)$. Hence $I=(6)$ is not a quasi $S$-$n$-ideal of $R$. Therefore, the converse of the Proposition \ref{2.3} is not holds in general.

\end{Example}



Now in the following example shows that an $S$-prime ideal need not be a quasi $S$-$n$-ideal of $R$. However, under certain conditions, every $S$-prime ideal is a quasi $S$-$n$-ideal of $R$.
\begin{Example}
    In $R=\mathbb{Z}^+_{0}$ as an entire semiring, let $I=(6)$ and $S=\{2^n : n\geq 0\}$ then $I$ is a $S$-prime ideal of $R$. Since for any $a,b \in R$, if $a b \in I$ then there exists $s \in S$ such that $s  a \in I $ or $s b \in I$. But $I$ is not a quasi $S$-$n$-ideal of $R$. Since $N(R)=\{0\}$, $2\cdot3=6 \in I$ but there does not exists any $s \in S$ such that $s\cdot 2^2 \in I $ or $s\cdot3 \in N(R)$.
\end{Example}
\begin{Theorem}
    Let $R$ be a commutative semiring and let $S \subseteq R$ be a multiplicatively closed subset. If $I$ is an $S$-prime ideal of $R$ and $I = N(R)$, then $I$ is a quasi $S$-$n$-ideal of $R$.
\end{Theorem}
\begin{proof}
    Let $I$ is an $S$-prime ideal of $R$, then for any $a,b \in R$ if $ab \in I=N(R)$ then there exists $s\in S $ such that $sa \in I=N(R) \;\text{or}\; sb \in N(R)$. Now if $sa \in I=N(R)$ then $sa^2=sa\cdot a\in I$. Then $I$ is a quasi $S$-$n$-ideal of $R$. If not then $ sb \in N(R)$. In both case $sa^2\in I \quad \text{or} \quad sb\in I$. Therefore $I$ is a quasi $S$-$n$-ideal of $R$.
\end{proof}
\begin{Theorem} \label{2.10}
Let $R$ be a commutative semiring and let $S \subseteq R$ be a multiplicatively closed set. If $I$ is a quasi $S$-$n$-ideal of $R$ and $\sqrt{I} = N(R)$, then $\sqrt{I}$ is an $S$-prime ideal of $R$.
\end{Theorem}

\begin{proof}
Assume that $I$ is a quasi $S$-$n$-ideal of $R$ and $\sqrt{I} = N(R)$. 
Let $a,b \in R$ such that $ab \in N(R) = \sqrt{I}$.
Then there exists $n \in \mathbb{N}$ such that $(ab)^n = a^n b^n \in I$.
Since $I$ is a quasi $S$-$n$-ideal, there exists $s \in S$ such that $a^n b^n \in I \;\Rightarrow\; s(a^n)^2 = s a^{2n} \in I \quad \text{or} \quad s b^n \in N(R)=\sqrt{I}$. If $s a^{2n} \in I $, then $sa \in \sqrt{I}= N(R)$. If $s b^n \in N(R)$, then $sb \in \sqrt{I}= N(R)$.
Therefore, $ab \in N(R) \;\Rightarrow\; sa \in N(R) \;\text{or}\; sb \in N(R)$.

Hence, $\sqrt{I}$ is an $S$-prime ideal of $R$. \end{proof}

Now we define  strongly quasi
$S$-primary ideal in commutative semiring similar as \cite{mo},  
 Let $I$ be a proper ideal of a commutative semiring $R$ and $S$ be a multiplicatively closed subset of $R$ satisfying $I\cap S=\phi$. Then $I$ is a strongly quasi $S$-primary ideal of $R$ if there exists an element $s\in S$ such that for all $a,b\in R$, $ab\in I \Longrightarrow sa^{2}\in I$
or  $sb\in\sqrt{I}$.

\begin{Theorem}

Let $R$ be a commutative semiring and let $S$ be a multiplicatively closed subset of $Reg(R)$. For a proper ideal $I$ of $R$ with $I\cap S=\phi$,

$I \text{ is a quasi } S\text{-}n\text{-ideal of }R
\quad\Longleftrightarrow\quad
I \text{ is a strongly quasi }S\text{-primary ideal of }R
\text{ and }\sqrt{I}=N(R)$.

\end{Theorem}

\begin{proof}
$(\Rightarrow)$
Let $a,b \in R$ with $ab \in I$ and suppose that $sb \notin \sqrt{I}$. Then $s^nb^{n} \notin I$ for all $n \in \mathbb{N}$, which implies $sb \notin N(R)\subseteq \sqrt{I}$. Hence, by the quasi $S$-$n$-ideal condition, we must have $sa^{2} \in I$. Thus $I$ is strongly quasi-S-primary. Next, we show that $\sqrt{I} = N(R)$.

 Let $x \in \sqrt{I}$. Then $x^k = 1 \cdot x^k \in I$. Since  $I$ is a quasi $S$-$n$-ideal of $R$, by Proposition \ref{2.3}, $s\sqrt I \subseteq N(R)$ then there exists $n \in \mathbb{N}$ such that $(s x)^n = 0$. Since $s$ is not a zero divisor, it follows that $x^{n} = 0$ and hence $x \in N(R)$. Therefore \( \sqrt{I} \subseteq N(R) \) and we know $N(R)\subseteq \sqrt{I}$.
Hence, $\sqrt{I} = N(R)$.

$(\Leftarrow)$
Conversely, suppose that $I$ is a strongly quasi-$S$-primary ideal of $R$ and $\sqrt{I}=N(R)$. Let $a,b\in R$ be such that $ab\in I$. Then there exists some $s\in S$ such that $ sa^{2}\in I \quad\text{or}\quad sb\in\sqrt{I}=N(R)$. Therefore, by the Definition~\ref{2.2}, $I$ is a quasi $S$-$n$-ideal of $R$.
 \end{proof}

\begin{Example}
Let $R=\mathbb{Z}_{12}^+$ and let $S=\{3^n : n\geq 0\}=\{1,3,9\}$. Then $S$ is multiplicatively closed, but $3$ is a zero divisor, so $S \nsubseteq \operatorname{Reg}(R)$.
Let $I=(4)=\{0,4,8\}$. Then $N(R)=\{0,6\},$ and $ \sqrt{I}=\{0,2,4,6,8,10\}$. Thus $\sqrt{I} \neq N(R)$. Now let $ab \in I$. Since $ab$ is even, either $a$ is even or $b$ is even. If $a$ is even, then $a^2$ is divisible by $4$, and hence $3a^2 \in I$. If $b$ is even, then $3b \in \{0,6\}=N(R)$. Thus $I$ is a quasi $S$-$n$-ideal of $R$, but $\sqrt{I} \neq N(R)$. Hence the equivalence fails when $S \nsubseteq \operatorname{Reg}(R)$.
\end{Example}
\begin{Definition}
    Let $R$ be a commutative semiring and let $S$ be a multiplicatively closed subset of $R$. A proper ideal $I$ of $R$ satisfying $I\cap S=\phi$, $I$ is called a $S$-$n$-ideal of $R$ if there exists an element $s\in S$ such that whenever $a,b\in R$ and $ab \in I$, either $sa\in I$ or $sb\in N(R)$.
\end{Definition}

\begin{Theorem}\label{10}
Let $R$ be a commutative semiring and let $S$ be a multiplicatively closed subset of $R$. If $I$ is a quasi $S$-$n$-ideal of $R$, then $\sqrt{I}$ is an $S$-$n$-ideal of $R$.
\end{Theorem}

\begin{proof}
Assume that $I$ is a quasi $S$-$n$-ideal of $R$. Then there exists $s\in S$ such that for all $a,b\in R$, whenever $ab\in I$, either $sa^{2}\in I$ or $sb\in N(R)$. Let $x,y\in R$ be such that $xy\in\sqrt{I}$. Then there exists a positive integer $n$ such that $(xy)^n=x^ny^n\in I$. Since $I$ is a quasi $S$-$n$-ideal, $sx^{2n}\in I \quad\text{or}\quad sy^n\in N(R)$. If $sx^{2n}\in I$ implies $(sx)^{2n}=s^{2n-1} \cdot sx^{2n}\in I$ therefore $sx\in\sqrt{I}$. If $sy^n\in N(R)$ implies $(sy)^{n}=s^{n-1} \cdot sy^{n}\in N(R)$ then there exits $m \in \mathbb{N}$ such that $(sy)^{nm} = 0$ therefore $sy\in N(R)$.

It remains to show that $\sqrt{I}\cap S=\phi$. If not, let $t\in\sqrt{I}\cap S$. Then $t^k\in I$ for some $k\in\mathbb{N}$. Since $S$ is multiplicatively closed and $t\in S$, we also have $t^k\in S$. Hence $t^k\in I\cap S$, contradicting the assumption that $I\cap S=\phi$.

 Thus whenever $xy\in\sqrt{I}$, we have $sx\in\sqrt{I}$ or $sy\in N(R)$.
Hence $\sqrt{I}$ is an $S$-$n$-ideal of $R$.\end{proof}

\begin{Corollary}
Let $I$ be a quasi $n$-ideal of a commutative semiring $R$. Then $\sqrt{I}$ is an $n$-ideal of $R$.\end{Corollary}

\begin{proof}
Taking $S=\{1\}$ in Theorem \ref{10}, the result follows.
\end{proof}

\begin{Proposition}\label{propquasi}
Every $S$-$n$-ideal of a semiring $R$ is a quasi $S$-$n$-ideal of $R$.
\end{Proposition}

\begin{proof}
Let $I$ be an $S$-$n$-ideal of $R$. Then there exists an element $s \in S$ such that for all $a,b \in R$, $ab \in I \;\Rightarrow\; sa \in I \ \text{or}\ sb \in N(R)$. If $sb \in N(R)$, then the required condition holds. If $sa \in I$, then $(sa)a = sa^{2} \in I$. Thus, $ab \in I \;\Rightarrow\; sa^{2} \in I \ \text{or}\ sb \in N(R)$, which shows that $I$ is a quasi $S$-$n$-ideal of $R$.
\end{proof}

\begin{Corollary}
 Let $R$ be a commutative semiring. Then every $n$-ideal of $R$ is a quasi $n$-ideal of $R$. \end{Corollary}
 
 However, we see later in the example \ref{3.2}, the converse does not hold in general.

Recall that, for an ideal $I$ of a commutative semiring $R$ and an element $r \in R$, then the set $(I:r)=\{x \in R: rx\in I\}$ is an ideal of $R$.
\begin{Proposition}
Let $R$ be a commutative semiring and let $S \subseteq \operatorname{reg}(R)$ be a multiplicatively closed subset of $R$. Let $I$ be a quasi $S$-$n$-ideal of $R$ if and only if $(I:s)$ is a quasi $n$-ideal of $R$ for some $s \in S$.
\end{Proposition}

\begin{proof}
 Since $I$ is a quasi $S$-$n$-ideal, there exists $s \in S$ such that for all $a,b \in R$, $ab \in I \;\Rightarrow\; s a^{2} \in I \ \text{or}\ s b \in N(R)$. 
 Let $a,b \in R$ such that $ab \in (I:s)$. Then $sab = a (sb)\in I$ implies $s(a)^2 \in I \ \text{or}\ s \cdot sb \in N(R)$. Now if $sa^2 \in I$ then $a^2 \in (I:s)$ and if $sa^2 \notin I$ then by defination $ s^2b \in N(R)$. Since $s \in S \subseteq \operatorname{reg}(R)$, it follows that $b \in N(R)$. Therefore, $ab \in (I:s) \;\Rightarrow\; a^2 \in (I:s) \ \text{or}\ b \in N(R)$, which shows that $(I:s)$ is a quasi $n$-ideal of $R$.

Converse, let $ab \in I$ that implies $ab \in (I:s) $. Since $(I:s)$ is an quasi $n$-ideal of $R$, $ab \in (I:s) \;\Rightarrow\; a^2 \in (I:s) \ \text{or}\ b \in N(R)$ which implies  $sa^2 \in I \ \text{or}\ b \in N(R)$ i,e. $sa^2 \in I \ \text{or}\ sb \in N(R)$, which shows that $I$ is an quasi
$S$-$n$-ideal of $R$.
\end{proof}

\begin{Theorem}
Let $R$ be a commutative semiring and let $S$ be a multiplicatively closed subset of $R$. 
Let $ \{I_i\}_{i\in \Delta}$ be 
a family of quasi $S$-$n$-ideals of $R$. Then $I = \displaystyle\bigcap_{i\in\Delta} I_i$ is a quasi $S$-$n$-ideal of $R$.
\end{Theorem}

\begin{proof}
Since each $I_i$ is a proper ideal of $R$, their intersection $I=\displaystyle\bigcap_{i\in\Delta} I_i$ is also a proper ideal of $R$. For each $i\in \Delta$, since $I_i$ is a quasi $S$-$n$-ideal, there exists an element $s_i \in S$ such that for all $a,b \in R$, $ab \in I_i \implies s_i a^{2} \in I_i \ \text{or}\ s_i b \in N(R)$. Set $s =\displaystyle\prod_{i\in\Delta} s_i \in S$, where we choose a common $s\in S$ dominating all $s_i$. Let $a,b \in R$ with $ab \in I$.
Then $ab \in I_i$ for every $i\in\Delta$. Fix $i\in\Delta$.
If $s_i b \in N(R)$ for some $i$, then since $s b =(\displaystyle\prod_{j\neq i}s_j)(s_i b)$ and $N(R)$ is an ideal of $R$, we obtain $sb \in N(R)$. Otherwise, $s_i a^{2} \in I_i$ for all $i\in\Delta$. Multiplying by $\displaystyle\prod_{j\neq i}s_j$, we get
$sa^{2} \in I_i \quad \text{for all } i\in\Delta$, and hence $sa^{2} \in \displaystyle\bigcap_{i\in\Delta} I_i = I$. Thus, $ab \in I \implies sa^{2} \in I \ \text{or}\ sb \in N(R)$, which shows that $I$ is a quasi $S$-$n$-ideal of $R$. \end{proof}

\begin{Proposition}


Let $I$ be a proper ideal of a commutative semiring $R$ such that $I \cap S= \phi$ where $S$ is a multiplicatively closed subset of $R$, then the following statements are equivalent: 

\item[(i)] $I$ is a quasi $S$-$n$-ideal of $R$.
\item[(ii)] For every $a \in R$, either $s(a) \subseteq (I:a)$ or $s(I:a) \subseteq N(R)$.
\item[(iii)] For any ideals $J$ and $K$ of $R$ with $JK \subseteq I$, either $sJ^{2} \subseteq I$ or $sK \subseteq N(R)$.
\item[(iv)] For $b \in R$, if $sb \notin N(R)$, $s(I:b)^{2} \subseteq I$.

\end{Proposition}

\begin{proof}

\noindent
$(i)\Rightarrow(ii)$.
Suppose that, $I$ is a quasi $S$-$n$-ideal. Take an element $a\in R$ and there exists $s\in S$ if $sa^{2}\in I $ then implies $s(a)\subseteq (I : a)$. Now if $sa^{2}\notin I$ then let $b\in (I:a)$ for some $b \in R$. Then $ab \in I$. Since $I$ is a quasi $S$-$n$-ideal of $R$ then  $ab \in I$  $\;\Rightarrow\; sa^{2}\in I \ \text{or}\ sb\in N(R)$ which implies $s(a) \subseteq (I:a)$ or $s(I:a) \subseteq N(R)$.

\medskip
\noindent $(ii)\Rightarrow(iii)$
Let $J$ and $K$ be an ideals of $R$ such that $JK\subseteq I$ and assume that  $sK \nsubseteq N(R)$, for all $s\in S$. Then there exists $k\in K\setminus N(R)$. For any $x\in J$, we have $xk\in JK\subseteq I$, and hence $k\in (I:x)$. Since $sk\notin N(R)$, we obtain $s(I:x)\nsubseteq N(R)$. By hypothesis $(ii)$, it follows that $s(x)\subseteq (I:x)$. Therefore $sx^{2}\in I$ for every $x\in J$ and hence $sJ^{2}\subseteq I$ or $sK \subseteq N(R)$.

\medskip
\noindent
$(iii)\Rightarrow(iv)$.
Let $sb \notin N(R)$, for all $s\in S$. Take $J=(I:b)$ and $K=(b)$. Since $b(I:b)\subseteq I$, we have $JK \subseteq I$. Because $sb \notin N(R)$, condition $(iii)$ implies $s(I:b)^{2} = sJ^{2} \subseteq I$ and hence $(iv)$ holds.

\medskip
\noindent
$(iv)\Rightarrow(i)$.
Let $a,b \in R$ such that $ab \in I$. Then $a \in (I:b)$. If $sb \in N(R)$, for some $s\in S$, then we are done. Otherwise, $sb \notin N(R)$ for all $s\in S$, then by (iv) we have $s(I:b)^{2} \subseteq I$, which implies $sa^{2} \in I$. Thus, $ab \in I  \implies  sa^{2} \in I \ \text{or} \ sb \in N(R)$, showing that $I$ is a quasi $S$-$n$-ideal of $R$.

Hence all four statements are equivalent. \end{proof}

\begin{Proposition}
Let $R$ be a commutative semiring and $S$ be a multiplicatively closed subset of $R$. Suppose that $I$ is a quasi $S$-$n$-ideal of $R$ and let $a\in R$ be such that $(a)=(a^{2})$. If $a\notin I$, then the ideal $(I:a)$ is a quasi $S$-$n$-ideal of $R$.
\end{Proposition}

\begin{proof}
Since $I$ is a quasi $S$-$n$-ideal of $R$, there exists $s\in S$ such that for all $x,y\in R$, whenever $xy\in I$, either $sx^{2}\in I$ or $sy\in N(R)$.
Let $x,y\in R$ be such that $xy\in (I:a)$ which implies $axy\in I$.
Applying the quasi $S$-$n$-ideal property of $I$ to $(ax)y\in I$, we obtain that either $s(ax)^{2}\in I $ or $sy\in N(R)$.
If $sy\in N(R)$, the desired condition holds. Suppose that this is not the case. Then $s a^{2}x^{2}\in I$. Since $(a)=(a^{2})$, there exists $r\in R$ such that $a=ra^{2}$. Multiplying by $r$, we get $sax^{2}\in I$. Hence $sx^{2}\in (I:a)$. Therefore, $xy\in (I:a) \Rightarrow sx^{2}\in (I:a) \ \text{or}\ sy\in N(R)$ and thus $(I:a)$ is a quasi $S$-$n$-ideal of $R$. \end{proof}

\begin{Proposition}
Let $f:R\to T$ be a homomorphism of commutative semirings and $S$ be a multiplicatively closed subset of $R$.Then the following statements hold:

(i) If $f$ is an epimorphism and $0 \notin f(S)$.  $ I$ is a quasi $S$-$n$-ideal of $R$ then $f(I)$ is a quasi $f(S)$-$n$-ideal of $T$.

(ii) If $f$ is monomorphisam and $I_0$ is a quasi $f(S)$-$n$-ideal of $T$, then $f^{-1}(I_0)$
is a quasi $S$-$n$-ideal of $R$.

\end{Proposition}

\begin{proof}
$(i)$ Assume that $f$ is surjective and $I$ is a quasi $S$-$n$-ideal of $R$.
Then there exists $s\in S$ such that for all $a,b\in R$, $ab\in I \Rightarrow sa^2\in I \ \text{or}\ sb\in N(R)$.
Let $x,y\in T$ such that $xy\in f(I)$. Since $f$ is surjective, there exist $a,b\in R$ with $f(a)=x$ and $f(b)=y$. Then
$xy=f(a)f(b)=f(ab)\in f(I)$, so $ab\in I$.
 By the quasi $S$-$n$ property of $I$, $sa^2\in I $ or $ sb\in N(R)$.
Then $f(sa^2)=f(s)f(a)^2\in f(I)$ or $ f(s)f(b)\in N(T)$.
Hence $f(s)x^2\in f(I)$ or $ f(s)y\in N(T)$, which shows that $f(I)$ is a quasi $f(S)$-$n$-ideal of $T$.

$(ii)$ Let $I_0$ be a quasi $f(S)$-$n$-ideal of $T$. Then there exists $s\in f(S)$ such that for all $x,y\in T$, $xy\in I_0 \Rightarrow sx^2\in I_0 \ \text{or}\ sy\in N(T)$.
Let $a,b\in R$ with $ab\in f^{-1}(I_0)$. Then $f(ab)=f(a)f(b)\in I_0$.
By the quasi $f(S)$-$n$ property of $I_0$,
$f(s)f(a)^2\in I_0 $ or $ f(s) f(b)\in N(T)$. Hence $f(sa^2)\in I_0 $ or $ f(sb)\in N(T)$. Thus $sa^2\in f^{-1}(I_0) $ or $ sb\in N(R)$, which proves that $f^{-1}(I_0)$ is a quasi $S$-$n$-ideal of $R$. \end{proof}

Let $I$ be a proper ideal of $R$. Let us define, $Z_I(R)=\{r\in R \mid rs\in I \text{ for some } s\in R\setminus I\}$, 
the set of zero-divisors of $R$ modulo $I$.
\begin{Theorem}\label{8}
Let $S \subseteq \operatorname{reg}(R)$ be a multiplicatively closed subset of a commutative semiring $R$, such that $N(S^{-1}R)=S^{-1}N(R)$ and let $I$ be a proper ideal of $R$.

If $I$ is a quasi $S$-$n$-ideal of $R$ iff  $S^{-1}I$ is a quasi $n$-ideal of $S^{-1}R$.

\end{Theorem}

\begin{proof}
 Since $I$ is a quasi $S$-$n$-ideal of $R$, there exists $s\in S$ such that for all $a,b\in R$, $ab\in I \Longrightarrow sa^{2}\in I \quad\text{or}\quad
sb\in N(R)$. Let $\frac{a}{u},\frac{b}{v}\in S^{-1}R$ be such that $ \frac{a}{u}\cdot\frac{b}{v}\in S^{-1}I$. Then there exists $t\in S$ satisfying $ tab\in I$. Since $I$ is a quasi $S$-$n$-ideal, we have $s(ta)^2=s t^2 a^2\in I$ or $sb\in N(R)$. Then $\frac{a^2}{u^2} = \frac{s t^2 a^2}{s t^2 u^2} \in S^{-1}I$ and $\frac{b}{v}=\frac{sb}{sv}\in S^{-1}N(R)$. Hence $\frac{a}{u}\frac{b}{v}\in S^{-1}I$ implies $\left(\frac{a}{u}\right)^2\in S^{-1}I \quad\text{or}\quad \frac{b}{v}\in N(S^{-1}R)$.
Therefore $S^{-1}I$ is a quasi $n$-ideal of $S^{-1}R$.

Conversely assume that $S^{-1}I$ is a quasi $n$-ideal of $S^{-1}R$. Let $a,b\in R$ with $ab\in I$. Then $ \frac{a}{1}\cdot\frac{b}{1}\in S^{-1}I$. Since $S^{-1}I$ is a quasi $n$-ideal of $S^{-1}R$, $ \left(\frac{a}{1}\right)^2\in S^{-1}I \quad\text{or}\quad \frac{b}{1}\in N(S^{-1}R)$. If $\left(\frac{a}{1}\right)^2\in S^{-1}I$, then there exists $s_1\in S$ such that $s_1a^{2}\in I$. If $ \frac{b}{1}\in N(S^{-1}R)$, then $ \frac{b}{1}\in S^{-1}N(R)$, and hence there exists $s_2\in S$ such that $ s_2b\in N(R)$. Let $s=s_1s_2\in S$. Then $ ab \in I$ implies $ sa^{2}\in I \quad\text{or}\quad sb\in N(R)$. Therefore $I$ is a quasi $S$-$n$-ideal of $R$.
\end{proof}
\begin{Corollary}
If $S^{-1}I$ is a quasi $n$-ideal of $S^{-1}R$ and $S\cap Z_I(R)=\phi=S\cap Z_{N(R)}(R)$, then $I$ is a quasi $n$-ideal of $R$.
\end{Corollary}

\begin{Definition} \cite{go}
Let $R$ be a commutative semiring, and let $I$ be a $Q$-ideal of $R$. Then there exists a subset $Q\subseteq R$ such that
$R=\displaystyle\bigcup_{q\in Q}(q+I)$ and for any $q_1,q_2\in Q$, $(q_1+I)\cap(q_2+I)\neq\phi
\Longleftrightarrow q_1=q_2$. The collection $R/I(Q)=\{q+I: q\in Q\}$
is called the quotient semiring of $R$ with respect to $I$. The operations on $R/I(Q)$ are defined by $(q_1+I)\oplus(q_2+I)=q_3+I$,
where $q_3\in Q$ is the unique element satisfying $q_1+q_2+I\subseteq q_3+I$
and $(q_1+I)\odot(q_2+I)=q_4+I$,
where $q_4\in Q$ is the unique element satisfying $q_1q_2+I\subseteq q_4+I$.
Then $(R/I(Q),\oplus,\odot)$ is the quotient semiring of $R$ by the $Q$-ideal $I$. Furthermore, if $q_0\in Q$ is the unique element satisfying
$q_0+I=I$, then $q_0+I$ is the zero element of $R/I(Q)$.

\end{Definition}
Let $R$ be a commutative semiring and consider the polynomial semiring
$R[x]$. Let $I=\langle x^m\rangle$
be a $Q$-ideal of $R[x]$. Then a corresponding partitioning set is
$$Q(x)=\left\{ a_0+a_1x+a_2x^2+ \cdots+a_{m-1}x^{m-1}  \mid a_i\in R \right\}.$$
Thus, every element of the quotient semiring $R[x]/\langle x^m\rangle(Q)$
has a unique representative in $Q(x)$. Then the quotient polynomial semiring is $R[x]/\langle x^m\rangle$ and let $\overline{S}=\{\,s+\langle x^m\rangle \mid s\in S\,\}\subseteq R[x]/\langle x^m\rangle$.
For any $s+\langle x^m\rangle,\; t+\langle x^m\rangle \in \overline{S}$ where $s,t \in S$, we have $(s+\langle x^m\rangle)(t+\langle x^m\rangle)=q+\langle x^m\rangle$, where $q \in Q(x)$ is the unique element satisfying $st+ I\subseteq q+I$. Since $s,t \in R$, their product $st$ is a constant polynomial in $R[x]$. Hence $st \in Q(x)$ and by uniqueness of $q$, we have $q=st$. Therefore $(s+\langle x^m\rangle)(t+\langle x^m\rangle)=st+\langle x^m\rangle$.
Since $S$ is multiplicatively closed set of $R$, so $st\in S$. Therefore, $st+\langle x^m\rangle \in \overline{S}$, which shows that $\overline{S}$ is multiplicatively closed in $R[x]/\langle x^m\rangle$. In the following theorem, we investigate the property of quasi $S$-$n$-ideal in the quotient polynomial semiring $R[x]/\langle x^m\rangle$.

\begin{Lemma}\label{11}
    For the quotient polynomial semiring $R[x]/\langle x^m\rangle$, the nilradical $N\!\left(R[x]/\langle x^m\rangle\right)=N(R)+Rx+\cdots+Rx^{m-1}+\langle x^m\rangle$.
\end{Lemma}
\begin{proof}


 Let $f=a_0+a_1x+a_2x^2+\cdots+a_{m-1}x^{m-1}+\langle x^m\rangle \in N\!\left(R[x]/\langle x^m\rangle\right)$. Then there exists $ n\in \mathbb{N}$ such that $f^n=0+\langle x^m\rangle $, also $f^n= q_0+\langle x^m\rangle $ where $q_0 \in Q(x)$ is the unique element satisfying $(\sum_{i=0}^{m-1}a_i x^i)^n +\langle x^m\rangle \subseteq q_0+\langle x^m\rangle=0+\langle x^m\rangle$. By the uniqueness of $q_0$, $q_0 = 0$ the zero polynomial in $Q(x)$. Therefore its constant coefficient is $0$, which gives $a_0^n=0$ and therefore $a_0 \in N(R)$, thus $f \in N(R)+Rx+\cdots+Rx^{m-1}+\langle x^m\rangle$. 
 
 Conversely, let $g=b_0+b_1x+b_2x^2+\cdots+b_{m-1}x^{m-1}+\langle x^m\rangle  \in N(R)+Rx+\cdots+Rx^{m-1}+\langle x^m\rangle$ then $b_0 \in N(R) $. Since $b_0^k=0$ for some $k \in \mathbb{N}$, there exists  $r \in \mathbb{N}$ such that $g^r=0+\langle x^m\rangle $, and also $g^r=q_0'+\langle x^m\rangle$  where $q_0' \in Q(x)$ is the unique element satisfying $(\sum_{i=0}^{m-1}b_i x^i)^r+ \langle x^m\rangle \subseteq q_0'+\langle x^m\rangle =0+\langle x^m\rangle$. By the uniqueness of $q_0'$, $q_0' = 0$, the zero polynomial in $Q(x)$. Therefore $g \in N\!\left(R[x]/\langle x^m\rangle\right)$.  Hence for the quotient polynomial semiring $R[x]/\langle x^m\rangle$, the nilradical $N\!\left(R[x]/\langle x^m\rangle\right)=N(R)+Rx+\cdots+Rx^{m-1}+\langle x^m\rangle$.
 \end{proof}
\begin{Theorem}\label{12}
Let $R$ be a commutative semiring, $S$ be a multiplicatively closed subset of $R$, and $I'$ a proper ideal of $R$ disjoint from $S$ and take $\overline{S}=\{\,s+\langle x^m\rangle \mid s\in S\,\}$. Then $J=I'+Rx+Rx^2+\cdots+Rx^{m-1}+\langle x^m\rangle$ be an ideal of $R[x]/\langle x^m\rangle$ then $\overline{S} \cap J= \phi$. Then $J$ is a quasi $\overline{S}$-$n$-ideal of $R[x]/\langle x^m\rangle$ if and only if $I'$ is a quasi $S$-$n$-ideal of $R$.
\end{Theorem}

\begin{proof}
Suppose that $I'$ is a quasi $S$-$n$-ideal of $R$. Let $f=a_0+a_1x+a_2x^2+ \cdots +a_{m-1}x^{m-1}+\langle x^m\rangle \quad\text{and}\quad g=b_0+b_1x+b_2x^2+ \cdots+b_{m-1}x^{m-1}+\langle x^m\rangle$, where $a_i,b_i \in R$, $0 \le i \le m-1$, be elements of $R[x]/\langle x^m\rangle$ such that $fg\in J$. Then the constant coefficient $a_0b_0\in I'$. Since $I'$ is a quasi $S$-$n$-ideal, there exists $s\in S$ such that $sa_0^2\in I' $ or $sb_0\in N(R)$. If $sa_0^2\in I'$, then the constant coefficient of $(s+\langle x^m\rangle)f^2$ belongs to $I'$. Hence $(s+\langle x^m\rangle)f^2\in J$. If $sb_0\in N(R)$, then $(s+\langle x^m\rangle)g \in N\!\left(R[x]/\langle x^m\rangle\right)$ by Lemma \ref{11}, since its constant coefficient $sb_0$ is nilpotent. Therefore $J$ is a quasi $\overline{S}$-$n$-ideal of $R[x]/\langle x^m\rangle$.

Conversely, assume that $J$ is a quasi $\overline{S}$-$n$-ideal of $R[x]/\langle x^m\rangle$. Let $a,b\in R$ with $ab\in I'$. Then $(a+\langle x^m\rangle)(b+\langle x^m\rangle)\in J$. Hence there exists $s\in S$ such that $(s+\langle x^m\rangle)(a+\langle x^m\rangle)^2\in J$ or $(s+\langle x^m\rangle)(b+\langle x^m\rangle) \in N\!\left(R[x]/\langle x^m\rangle\right)$. Therefore the constant coefficients $sa^2\in I'$
 or $sb\in N(R)$ by Lemma \ref{11}, showing that $I'$ is a quasi $S$-$n$-ideal of $R$.
\end{proof}

\begin{Corollary}
    If $S=\{1\}$, then $J$ is a quasi $n$-ideal of $R[x]/\langle x^m\rangle$ if and only if $I'$ is a quasi $n$-ideal of $R$.
\end{Corollary}
\begin{Example}
    
Let $R=\mathbb{Z}_{12}^+,\quad I'=\langle 4\rangle=\{0,4,8\}$, and let $S=\{3^n:n\geq0\}$. Then $I'\cap S=\phi$ and $N(R)=\langle6\rangle=\{0,6\}$. Choosing $s=3\in S$, now whenever $ab\in I'$, either $3a^{2}\in I' \quad\text{or}\quad 3b\in N(R)$. Hence $I'$ is a quasi $S$-$n$-ideal of $R$. Now let $I=\langle x^{3}\rangle$. Then $R[x]/\langle x^{3}\rangle$ is the quotient polynomial semiring, with partitioning set $Q=\{ a_0+a_1x+a_2x^2\}$ and $J=I'+Rx+Rx^{2}+\langle x^{3}\rangle =\langle4\rangle+Rx+Rx^{2}+\langle x^{3}\rangle$. Also, $\overline{S} =\{\,3^{n}+\langle x^{3}\rangle:n\geq0\,\}$.

By the Theorem  \ref{12}, $J$ is a quasi $\overline{S}$-$n$-ideal of $R[x]/\langle x^{3}\rangle$.

\end{Example}

\section{Idealizations and Amalgamations}\label{2}



Let $R$ be a commutative semiring and $M$ an $R$-semimodule. The idealization of $M$ over $R$, denoted by $R(+)M$, is the semiring whose underlying set is $R\times M$, with componentwise addition and multiplication defined by
\[
(r_1,m_1)(r_2,m_2)=(r_1r_2,r_1m_2+r_2m_1).
\]
The idealization construction and its properties in the context of $S$-$k$-primary ideals of semirings were studied by Bhowmick and Goswami \cite{bs}.

Suppose that $I$ is an ideal of $R$ and $N$ is a subsemimodule of $M$. Then $I(+)N=\{(i,n)\mid i\in I, n\in N\}$ is an ideal of $R(+)M$ if $IM\subseteq N$. Under this condition, its radical is given by $\sqrt{I(+)N}=\sqrt{I}(+)M.$
As a consequence, we have $N(R(+)M)=N(R)(+)M$.

Moreover, if $S$ is a multiplicatively closed subset of $R$, then $S(+)M=\{(s,m)\mid s\in S,\ m\in M\}$ and $S(+)\{0\}=\{(s,0)\mid s\in S\}$
are multiplicatively closed subsets of $R(+)M$.

In this section, we investigate the relationship between quasi $S$-$n$-ideals of $R$ and quasi $S(+)M$-$n$-ideals of the idealization $R(+)M$. In particular, we establish results describing how the quasi $S$-$n$-ideal property is transferred between a semiring and its idealization.

\begin{Proposition} \label{3.1}
 Let $R$ be a commutative semiring and $N$ be a subsemimodule of $M$. Let $S$ be a multiplicatively closed subset of $R$ and $I$ be an ideal of $R$ such that $I \cap S = \phi$.

 $(i)$ If $I(+)N$ is a quasi $S(+)M$-$n$-ideal of $R(+)M$, then $I$ is a quasi $S$-$n$-ideal of $R$.

 $(ii)$ Suppose that $N(R)M\subseteq N$. If $I$ is a quasi $S$-$n$-ideal of $R$, then $I(+)N$ is a quasi $S(+)M$-$n$-ideal of $R(+)M$.
\end{Proposition}

\begin{proof} $(i)$
Since $I \cap S = \phi$ then $I(+)N \cap S(+)M= \phi$. Since $I(+)N$ is a quasi $S(+)M$-$n$-ideal of $R(+)M$, there exists $(s,m)\in S(+)M$ such that for any $(a,x),(b,y)\in R(+)M$, $(a,x)(b,y)\in I(+)N$ implies $(s,m)(a,x)^2\in I(+)N \quad \text{or} \quad (s,m)(b,y)\in N(R(+)M)$, where $N(R(+)M)$ denotes the nilradical of $R(+)M$. To show that $I$ is a quasi $S$-$n$-ideal of $R$, let $a,b\in R$ be such that $ab\in I$. Then $(a,0)(b,0)=(ab,0)\in  I(+)N$. Since $I(+)N$ is a quasi $S(+)M$-$n$-ideal, it follows that $(s,m )(a,0)^2\in I(+)N \quad \text{or} \quad (s,m)(b,0)\in N(R(+)M)$. Now, $(a,0)^2=(a^2,0)$, and hence $(s,m)(a^2,0)=(sa^2,a^2m)$. If $(sa^2,a^2m)\in I(+)N$, then necessarily $sa^2\in I$. On the other hand, $(s,m)(b,0)=(sb,bm)$. If $(sb,bm)\in N(R(+)M)$, then $sb\in N(R)$, since the nilradical of the idealization satisfies $N(R(+)M)=N(R)(+)M$.

Therefore, $ab\in I \implies sa^2\in I \ \text{or}\ sb\in N(R)$. Hence $I$ satisfies the defining condition of a quasi $S$-$n$-ideal of $R$. Therefore $I$ is a quasi $S$-$n$-ideal of $R$.

    $(ii)$ 
Let $R$ be a commutative semiring, $M$ be an $R$-semimodule, $S$ be a multiplicatively closed subset of $R$, and $N$ be a subsemimodule of $M$ such that $ N(R)M\subseteq N$. Since $I$ is a quasi $S$-$n$-ideal of $R$, there exists an element $s\in S$ such that whenever $a,b\in R$ and $ab\in I$, then $sa^2\in I \quad \text{or} \quad sb\in N(R)$. Now we show that  $I(+)N$ is a quasi $S(+)M$-$n$-ideal of $R(+)M$.

Let $(a,x),(b,y)\in R(+)M$ be such that $(a,x)(b,y)\in I(+)N$.
Then $(ab,ay+bx)\in I(+)N$, and consequently $ab\in I$. Since $I$ is a quasi $S$-$n$-ideal of $R$, it follows that $sa^2\in I \quad \text{or} \quad sb\in N(R)$. Now, $(a,x)^2=(a^2,ax+ax)$ and then $(s,0)(a,x)^2 =(sa^2,2sax)$. If $sa^2\in I$, since  $sI \subseteq N(R)$ which implies $sa \in N(R)$, Therefore $2sax\in N(R)M \subseteq N $. Thus we obtain $(s,0)(a,x)^2\in I(+)N$. On the other hand, if $sb\in N(R)$, then there exists $n\in\mathbb{N}$ such that $(sb)^n=0$. Now $(s,0)(b,y)=(sb,sy)$. Thus $(sb,sy)$ is nilpotent in $R(+)M$, and therefore $(s,0)(b,y)\in N(R(+)M)=N(R)(+)M$.

Hence, whenever $(a,x)(b,y)\in I(+)N$, we have either $(s,0)(a,x)^2\in I(+)N$ or $(s,0)(b,y)\in N(R(+)M)$. Since $(s,0)\in S(+)M$, it follows that $I(+)N$ is a quasi $S(+)M$-$n$-ideal of $R(+)M$. \end{proof}

\begin{Example}\label{3.2} In this example we see that $I(+)N$ is a quasi $S(+)M$-$n$-ideal of $R(+)M$ but $I(+)N$ is not a $S(+)M$-$n$-ideal of $R(+)M$.

Let $ R=\mathbb Z^+$, $ M=\mathbb Z_{12}^+,$ $ S=\{2^n
: n \in \mathbb N_0\} $ and let $ I =(0)$ and $N={\bar 0}$. However, $0(+)\bar 0$ is a quasi $(S(+)\mathbb Z_{12}^+)$-$n$-ideal $\mathbb Z^+(+)\mathbb Z_{12}^+$  but not a $(S(+)\mathbb Z_{12}^+)$-$n$-ideal of $\mathbb Z^+(+)\mathbb Z_{12}^+$ . Since, $ (0,\bar 4)(3,\bar 0)=(0,\bar 0)\in 0(+)\bar 0$ . Let $(s,m)\in S(+)\mathbb Z_{12}^+$, then $(s,m)(0,\bar 4)= (0,\bar {4s})\neq (0,\overline {0})$. So, $(s,m)(0,\bar 4) \notin 0(+)\bar 0$. On the other hand, $(s,m)(3,\bar 0)=(3s,\bar {3m})$. We have $3s\neq 0$. Consequently, $(s,m)(3,\bar 0)\notin N(\mathbb Z^+(+)\mathbb Z_{12}^+$). Thus, for $(s,m)\in S(+)\mathbb Z_{12}^+$, $(s,m)(0,\bar 4)\notin 0(+)\bar 0$ and  $(s,m)(3,\bar 0)\notin N(\mathbb Z^+(+)\mathbb Z_{12}^+)$. Since $(3,\bar 0)(0,\bar 4)\in 0(+)\bar 0$, the defining condition of a $(S(+)\mathbb Z_{12}^+)$-$n$-ideal fails. Therefore, $0(+)\bar 0$ is not a $(S(+)\mathbb Z_{12}^+)$-$n$-ideal of $\mathbb Z^+(+)\mathbb Z_{12}^+$.

Since $ I =(0)$ is a quasi $S$-$n$-ideal of $ R=\mathbb Z^+$ and $N(R)M\subseteq N$. So by the Proposition \ref{3.1},  $I(+)N$ is a quasi $S(+)M$-$n$-ideal of $R(+)M$.  We have $ (0,\bar 4)^2=(0,\overline {0})$, and therefore $(s,m)(0,\bar 4)^2=(0,\overline {0})\in 0(+)\bar 0$.
\end{Example}

\begin{Remark}
    Since any element of the form $(0,\bar a) \in Z^+(+)\mathbb Z_{n}^+$, where $ \bar a \in \mathbb Z_{n}^+$ is a nilpotent element of $ Z^+(+)\mathbb Z_{n}^+$. 
\end{Remark}

Let $R_1$ and $R_2$ be two commutative semirings with identity, let $J$ be an ideal of $R_2$, and let $f:R_1\longrightarrow R_2$ be a semiring homomorphism. We define $R_1\bowtie^{f}J = \{(a,f(a)+j)\mid a\in R_1,\; j\in J\}$. It is easy to verify that $R_1\bowtie^{f}J$ is a subsemiring of the direct product semiring $R_1\times R_2$. This semiring is called the amalgamation of $R_1$ with $R_2$ along the ideal $J$ with respect to $f$. This construction was introduced and studied by D'Anna et al.~\cite{an, an1} in the context of commutative rings.

Let $I$ be an ideal of $R_1$. Corresponding to $I$, we define the following ideals of $R_1\bowtie^{f}J$: $I\bowtie^{f}J = \{(i,f(i)+j)\mid i\in I,\; j\in J\}$. Now let $S$ be a multiplicatively closed subset of $R_1$. Then $S\bowtie^{f}J = \{(s,f(s)+j)\mid s\in S,\; j\in J\}$ and $S'=\{(s,f(s))\mid s\in S\}$ are multiplicatively closed subsets of $R_1\bowtie^{f}J$. Furthermore, if $0\notin f(S)$, then $f(S)=\{f(s)\mid s\in S\}$ is a multiplicatively closed subset of $R_2$.







\begin{Theorem} \label{t3.4}
Let $R\bowtie^{f}J$ be the amalgamation of commutative semirings $R$ and $R'$ along the ideal $J$ of $R'$ with respect to a semiring homomorphism $f:R\to R'$. Let $S$ be a multiplicatively closed subset of $R$ and let $I$ be an ideal of $R$ with $I\cap S=\phi$. Then the following assertions hold:

$(i)$ If $I\bowtie^{f}J$ is a quasi $S'$-$n$-ideal of $R\bowtie^{f}J$, then $I\bowtie^{f}J$ is a quasi $(S\bowtie^{f}J)$-$n$-ideal of $R\bowtie^{f}J$.

$(ii)$ If $I\bowtie^{f}J$ is a quasi $(S\bowtie^{f}J)$-$n$-ideal of $R\bowtie^{f}J$, then $I$ is a quasi $S$-$n$-ideal of $R$.
\end{Theorem}

\begin{proof}
$(i)$ Since $S'\subseteq S\bowtie^{f}J$, every quasi $S'$-$n$-ideal of $R\bowtie^{f}J$ is a quasi $(S\bowtie^{f}J)$-$n$-ideal of $R\bowtie^{f}J$. This proves $(1)$. 

$(ii)$ Let $x,y\in R$ such that $xy\in I$. Then $(x,f(x))(y,f(y))=(xy,f(xy)) \in I\bowtie^{f}J$.
Since $I\bowtie^{f}J$ is a quasi $(S\bowtie^{f}J)$-$n$-ideal of $R\bowtie^{f}J$, there exists a fixed element $(s,f(s)+j)\in S\bowtie^{f}J$, where $s \in S$ and $j \in J$, such that 
$$(s,f(s)+j)(x,f(x))^2\in I\bowtie^{f}J \text{  or  } (s,f(s)+j)(y,f(y))\in N(R\bowtie^{f}J).$$

If $(s,f(s)+j)(x,f(x))^2\in I\bowtie^{f}J$, then $(sx^2,f(sx^2)+jf(x^2)) \in I\bowtie^{f}J$, which implies that $sx^2\in I$. If $(s,f(s)+j)(y,f(y))\in N(R\bowtie^{f}J)$, then $(sy, f(sy)+jf(y)) \in N(R\bowtie^{f}J)$. The canonical projection $\pi:R\bowtie^{f}J\longrightarrow R$, defined by $\pi(r,f(r)+j)=r$, is a semiring homomorphism, it follows that $sy\in N(R)$. Hence, there exists $s\in S$ such that $sx^2\in I$ or $sy\in N(R)$. Therefore, $I$ is a quasi-$S$-$n$-ideal of $R$. \end{proof}

\begin{Theorem}
Let $R\bowtie^{f}J$ be the amalgamation of commutative semirings $R$ and $R'$ along the ideal $J$ of $R'$ with respect to a semiring homomorphism $f:R\to R'$. Let $S$ be a multiplicatively closed subset of $R$, and let $I$ be an ideal of $R$ such that $I\cap S=\phi$ and $J \subseteq N(R') $. Then the following assertions are equivalent: 

$(i)$ $I\bowtie^{f}J$ is a quasi $S'$-$n$-ideal of $R\bowtie^{f}J$;

$(ii)$ $I\bowtie^{f}J$ is a quasi $(S\bowtie^{f}J)$-$n$-ideal of $R\bowtie^{f}J$;

$(iii)$ $I$ is a quasi $S$-$n$-ideal of $R$. 
\end{Theorem}

\begin{proof}
$(i) \implies (ii)$ follows from Theorem \ref{t3.4} $(i)$ and $(ii) \implies (iii)$ follows from Theorem \ref{t3.4} $(2)$.

$(iii) \implies (i)$: Let  $(x,f(x)+j_1)(y,f(y)+j_2)=(xy,(f(x)+j_1)(f(y)+j_2)) \in I\bowtie^{f}J$, where $x,y \in R$ and $j_1, j_2 \in J$. Then $xy \in I$. Since $I$ is a quasi $S$-$n$-ideal of $R$, there exists an fixed element $s \in S$ such that, whenever $a,b \in R$ and $ab \in I$, either $sa^2 \in I$ or $sb \in N(R)$. Therefore $sx^2 \in I$ or $sy\in N(R)$. Consider the element $(s,f(s)) \in S'$.
If $sx^2 \in I$, then $(s,f(s))(x,f(x)+j_1)^2=((s,f(s))(x^2,f(x^2)+2j_1f(x)+j^2_1)=(sx^2,f(sx^2)+j)$ for some $j \in J$. Hence,  $(s,f(s))(x,f(x)+j_1)^2 \in I\bowtie^{f}J$. Suppose that $sy\in N(R)$. Since $f$ is a semiring homomorphism, $f(sy)\in N(R')$. Moreover, $j_2\in J\subseteq N(R')$, and hence $f(s)j_2\in N(R')$. Since $R'$ is commutative, it follows that $f(sy)+f(s)j_2\in N(R')$.
Consequently, $(s,f(s))(y,f(y)+j_2)
=(sy,f(sy)+f(s)j_2) \in N(R\bowtie^{f}J)$.

Thus, $(s,f(s))(x,f(x)+j_1)^2 \in I\bowtie^{f}J$ or $(s,f(s))(y,f(y)+j_2) \in N(R\bowtie^{f}J)$. Hence, $I\bowtie^{f}J$ is a quasi $S'$-$n$-ideal of $R\bowtie^{f}J$.
\end{proof}


\vspace{0.5cm}
\noindent
\textbf{Funding.}
The first author gratefully acknowladges the financial support from the Council of Scientific and Industrial Research (CSIR), India under File No. 09/0096(20069)/2024-EMR-I.

\vspace{0.5cm}
\noindent
\textbf{Conflict of interest.} All authors declare that they have no conflicts of interest.

\end{document}